\documentclass[10pt,reqno,twoside]{amsart}

\DeclareFontFamily{OT1}{rsfs}{}
\DeclareFontShape{OT1}{rsfs}{n}{it}{<-> rsfs10}{}
\DeclareMathAlphabet{\mathscr}{OT1}{rsfs}{n}{it}

\newtheorem{theorem}{Theorem}[section]
\newtheorem{lemma}[theorem]{Lemma}
\newtheorem{corollary}[theorem]{Corollary}
\newtheorem{remark}[theorem]{Remark}

\numberwithin{equation}{section}

\renewcommand{\L}{\mathscr{L}}

\newcommand{\Q}{{\mathbb{Q}}}

\renewcommand{\Re}{{\mathfrak{Re}}}
\renewcommand{\Im}{{\mathfrak{Im}}}

\newcommand{\s}{{\sigma}}

\renewcommand{\b}{\beta}
\newcommand{\g}{\gamma}

\newcommand{\myref}[1]{(\ref{#1})}

\reversemarginpar 

\def\ds{\displaystyle}

\begin{document}
\title[Explicit zero-free regions for Dedekind Zeta functions]{Explicit zero-free regions for Dedekind Zeta functions}
\author{Habiba Kadiri}
\address{Mathematics and Computer Science Department, University of Lethbridge, 4401 University Drive, Lethbridge, Alberta, T1K 3M4 Canada}
 \email{habiba.kadiri@uleth.ca}
\date{\today}
\begin{abstract}
Let $K$ be a number field, $n_K$ its degree, and $d_K$ the absolute value of its discriminant. 
We prove that, if $d_K$ is sufficiently large, then the Dedekind zeta function $\zeta_K(s)$ has no zeros in the region: 
$ \Re s \ge 1- \frac{1}{\log M}$, $|\Im s| \ge 1$, where $\log M= 12.55 \log d_K + 9.69 n_K \log|\Im s| + 3.03 n_K + 58.63$. 
Moreover, it has at most one zero in the region:
$ \Re s \ge 1- \frac{1}{12.74\log d_K}$, $|\Im s| \le 1. $
This zero if it exists is simple and is real.
This argument also  improves a result of Stark by a factor of 2: $\zeta_K(s)$ has at most one zero in the region $ \Re s \ge  1  -   \frac{1}{2\log d_K}$, $|\Im s|   \le     \frac{1}{2\log d_K}. $
\end{abstract}
\maketitle 
\section{Introduction}\noindent
Let $K$ be a number field. Its degree is denoted $n_K = [K : \mathbb{Q}]$, the absolute value of its discriminant is $d_K$, and the Dedekind zeta function associated to $K$ is $\zeta_K(s)$.  
In this article, we prove an explicit classical zero-free region for $\zeta_K(s)$.

Rosser and Schoenfeld published a series of articles devoted to obtaining improved estimates for prime counting functions (see \cite{Ros}, \cite{RS1}, \cite{RS2} and \cite{Sch}),
enlarging de La Vall\'ee Poussin's classical zero-free region in \cite{RS2}.
By employing the global explicit formula for $-\frac{\zeta'}{\zeta}(s)$ and building on an argument of Stechkin \cite{St1}, they proved that $\zeta(s)$ has no zeros in the region
\begin{equation}\label{zfr-RS}
\Re s \ge 1-\frac1{R_1 \log(|\Im s|/17)}
\end{equation}
where $R_1=9.645908801$. 
McCurley applied the same method to Dirichlet $L$-functions. 
He proved in \cite{MC2} that the product $\mathcal{L}_q(s)=\prod_{\chi \bmod q} L(s,\chi)$ has at most a single zero in the region 
\begin{equation}\label{zfr-MC}
\Re s \ge 1-\frac1{R_2\log \max \left(q,q|\Im s|,10\right)}
\end{equation}
where $R_2=9.645908801$.  
The single zero, if it exists is real, simple, and corresponds to a non-principal real character.
The constant is independent of the modulus $q$, and is valid for any value of $q \ge 3$. Observe that the two constants agree: $R_2=R_1$. 

In $1992$, Heath-Brown established an asymptotic result which provides a wider zero-free region for sufficiently large modulus $q$:
\begin{equation}\label{zfr-HB1}
\Re s \ge 1-\frac{1}{R_3 \log q}, \quad |\Im s| \le 1
\end{equation}
where $R_3 =2.8735\ldots $ is smaller than McCurley's constant. 
Heath-Brown's method is different from the one used to obtain \myref{zfr-RS} and \myref{zfr-MC}. 
Some of the main tools in his proof are: a smooth explicit formula for zeros of Dirichlet $L$-functions, Burgess' sub-convexity bound
for Dirichlet $L$-functions, and a local Jensen type formula. 
The previous zero-free region is one of the main ingredients in the proof of Linnik's theorem on the size of the smallest prime $P(a,q)$ in an arithmetic progression ($a$ modulo $q$).
In his groundbreaking article \cite{HB}, Heath-Brown improves drastically all previous results on Linnik's theorem and shows that $P(a,q)\ll q^{5.5+\epsilon}$. 
Recently, Xylouris reduced $R_3$ to $2.2727\ldots $ and Linnik's constant to $5.2$ in his Ph.D. thesis \cite{Xyl}.
In 2000, Ford \cite{For} applied Heath-Brown's argument to the case of the Riemann zeta function. This allowed him to produce an explicit Korobov-Vinogradov zero-free region and to  widen the region in \myref{zfr-RS}  by replacing $R_1$ by $8.463$.

In \cite{Kad1}, the author further reduced the value of $R_1$ to $5.69693$. 
The method used a global explicit formula applied to  a smoothed version of the Riemann zeta-function, together with a generalization of Stechkin's lemma.
This method also improves McCurley's result.
In \cite{Kad2}, the author finds that $R_2=6.50$ is an admissible value for any Dirichlet $L$-function.

In comparison, in the number field setting, there are no analogous theorems to \myref{zfr-RS}, \myref{zfr-MC}, and \myref{zfr-HB1} with explicit constants. 
However, in \cite{Sta} Stark established an explicit result for the Dedekind zeta function in a restricted region.
He established that for any number field $K\not=\Q$, $\zeta_K(s)$ has at most one zero in the region
\begin{equation}\label{zfr-Sta}
\Re s \ge 1-\frac1{4\log d_K},\quad |\Im s| \le \frac1{4\log d_K}.
\end{equation}
If such a zero exists, it is real and simple.
In  Lemma 2.3 of \cite{LMO}, Lagarias, Montgomery, and Odlyzko establish zero-free regions for Hecke $L$-functions. 
They prove that for all finite extensions $K$ of $\Q$ and Hecke characters $\chi$ on $K$, the Hecke $L$-function $L(s,\chi,K)$ has at most one zero in the region
\begin{equation}\label{zfr-LMO1}
\Re s > 1-\frac1{R_4\log A(\chi)}, \quad |\Im s| < \frac1{R_4\log A(\chi)}, 
\end{equation}
where $A(\chi)=d_K N_{K/\Q}f(\chi)$, $f(\chi)$ being the conductor of $\chi$, and where $R_4$ is a positive constant, independent of $K$.
They also extend the region to the whole critical strip and prove that $L(s,\chi,K)$ has no zeros in  the region
\begin{equation}\label{zfr-LMO2}
\Re s > 1-\frac1{R_4(\log A(\chi)+n_K \log(|\gamma|+2))}.
\end{equation}
The classical argument of de La Vall\'ee Poussin is used to prove the above inequalities.
Moreover, \myref{zfr-LMO1} and \myref{zfr-LMO2} play an important role in their proof of a bound for the least prime ideal in the Chebotarev density theorem.
However, the constant $R_4$ is not made explicit. 
In this article we shall apply some of the above mentioned techniques to obtain the following result.
\begin{theorem}\label{ZFR}
Let $d_K$ be sufficiently large. Then $\zeta_K(s)$ has no zero in the region:
\begin{equation}\label{zfrD1}
\Re s \ge 1- \frac{1}{ 12.55 \log d_K + 9.69 ( \log|\Im s|  ) n_K + 3.03 n_K + 58.63} ,\quad |\Im s| \ge 1. 
\end{equation}
Moreover, $\zeta_K(s)$ has at most one zero in the region:
\begin{equation}\label{zfrD2}
\Re s \ge 1- \frac{1}{12.74\log d_K},\quad |\Im s| \le 1. 
\end{equation}
This zero if it exists is simple and is real.
\end{theorem}\noindent
An improvement of Stark's result \myref{zfr-Sta} follows from the method leading to \myref{zfrD2}:
\begin{corollary}\label{Stark}
Let $d_K$ be sufficiently large. Then $\zeta_K(s)$ has at most one zero in the region:
\begin{equation}\label{zfrD3}
\Re s \ge 1- \frac{1}{2\log d_K},\quad |\Im s| \le \frac{1}{2\log d_K}.
\end{equation}
Also, we can prove that $\zeta_K(s)$ has at most one zero in the region:
\begin{equation}\label{zfrD4}
\Re s \ge 1- \frac{1}{1.70\log d_K},\quad |\Im s| \le \frac{1}{4\log d_K}.
\end{equation}
This zero if it exists is simple and is real.
\end{corollary}\noindent
Note that the above theorems can be made completely explicit for any value of $\log d_K$.

Our proof does not make use of Heath-Brown's version of Jensen's formula.  We now explain why it appears difficult 
to apply his approach to the number field setting. He proved that for $\chi$ a non-principal Dirichlet character modulo $q$,  $ \sigma $ close to $1$, and for any 
$\epsilon >0$, there exists a $\delta_{\epsilon} >0$ such that 
\[
    - \Re \frac{L'}{L}(\sigma ,\chi) \le - \sum_{|1-\varrho| \le \delta_{\epsilon} }  \Re \frac{1}{\s-\varrho} + 
    \Big( \frac{\phi}{2} + \epsilon \Big) \log q.
\]
Here $\phi$ is a constant associated to an upper bound for $L(s,\chi)$.   The convexity bound
yields $\phi= \frac{1}{2}$ and Burgess' sub-convexity estimate yields $\phi=\frac{1}{3}$. 
In comparison, this method applied in the context of number fields leads to the following inequality:
\[
    - \Re \frac{\zeta_K'}{\zeta_K}(\sigma ) \le   \frac1{\s-1} 
    + \Big(\frac14+\epsilon\Big)\log d_K 
    + 2 n_K \log \Big( \frac{\log d_K}{n_K} \Big) + \mathcal{O}(n_K).  
\]
 (see Lemma 4 in Li's article \cite{Li} for a reference).
The problem here is that the error terms may become larger than the main term which is of size $\log d_K$. 
Note that the coefficient $\frac14$ follows from the convexity bound. Sub-convexity bounds for number fields only have been proven for some special cases, such as cubic extensions (see pp. 54-55 of \cite{Ven} for an overview of known results). 
On the other hand, Stechkin's argument leads to the inequality: 
\[
    - \Re \frac{\zeta_K'}{\zeta_K}(\sigma ) \le   \frac1{\s-1}  - \sum_{\varrho}  \Re \frac{1}{\s-\varrho}  + \Big(\frac{1-\frac1{\sqrt5}}2+\epsilon \Big)\log d_K ,
\]
where $\frac{1-\frac1{\sqrt5}}2 = 0.27639\ldots $.
Moreover, the above is valid for any number field $K$.

The Dedekind zeta function is similar to the Riemann zeta function as it also has a simple pole at $s=1$.
However, a major difference is that its zeros can lie very close to the the real axis. From this perspective, it behaves similarly to a Dirichlet $L$-function. 
As a consequence, our argument is closer to McCurley's than to Rosser and Schoenfeld's.
Note that our coefficient of $\log d_K$ is larger than their coefficient of $\log q$ (namely $12.55$ instead of $9.645908801$).
This is due to the fact that the zero-free region proof compares values of $\zeta_K(s)$ at different points $s$ close to the $1$-line. For each of these points, a contribution of $\log d_K$ arises, even when $s$ is real. 
On the other hand, this does not occur in the case of Dirichlet $L$-functions: there is no contribution of $\log q$ from $L(s,\chi)$, when $s$ and $\chi$ are real.

One of the interests of Theorem \ref{ZFR} is its application to the problem of finding an explicit upper bound for the least prime ideal in the Chebotarev density theorem. 
Given a Galois extension of number fields $E/K$ with group $G$ and a conjugacy class $C \subset G$, there exists an unramified prime 
ideal $\mathfrak{p}$ of degree one such that $\sigma_{\mathfrak{p}}=C$ and $\mathbb{N} \mathfrak{p} \le d_K^{C_0}$ for an explicit constant $C_0 >0$.
\section{Notation and preliminary lemmas}\label{def-not}\label{preliminary-lemmas}
Let $K$ be a number field with ring of integers $\mathcal{O}_K$. 

The Dedekind zeta function of $K$ is  
\[
   \zeta_K(s) =
\sum_{\substack{\mathfrak{a} \subset  \mathcal{O}_K \\ \mathfrak{a} \ne 0}}
\frac{1}{(\mathbb{N} \mathfrak{a})^s},\ \text{ for }\ \Re (s) >1.
\]
It possesses the Euler product
\[
  \zeta_K(s) = \prod_{\mathfrak{p}} (1-(N \mathfrak{p})^{-s})^{-1}
\]
where $\mathfrak{p}$ ranges over all prime ideals in $\mathcal{O}_K$ and
$\Re(s) >1$.
It is convenient to consider the completed zeta function
\begin{align*}  \label{def-xi}
\xi_K(s) & =s(s-1) {d_K}^{s/2}    \gamma_K(s)\zeta_K(s)
   \ ,   \\
\gamma_K(s)  & = 2^{r_2}( \pi^{n_K}2^{2r_2} )^{-s/2} \Gamma(s/2)^{r_1}
\Gamma(s)^{r_2}
\end{align*}
where $r_1$ and $r_2$ are the number of real and complex places in $K$.
The advantage of $\xi_K$ is that it is an entire function which satisfies the functional equation:
\begin{equation*} 
  \label{eq:xife}
  \xi_K(s)=\xi_K(1-s).
\end{equation*}
By the duplication formula $\Gamma(s)=	\frac{2^{s-1}}{\sqrt{\pi}}\Gamma(\frac{s}{2})\Gamma(\frac{s+1}{2})$, it follows that
\begin{equation}\label{eq:gamK}
  \gamma_K(s) = \pi^{-\frac{b(s+1)}{2}} \Gamma\left(\tfrac{s+1}{2}\right)^{b} \pi^{-\frac{as}{2}} \Gamma\left(\tfrac{s}{2}\right)^{a} ,
\end{equation}
where $a,b$ are integers which satisfy $a+b=n_K$. 

Let $\sigma>1$ and $t$ real.
We shall use the following notation and assumptions throughout the rest of the article:
\begin{eqnarray}\label{eq:cond} 
& \L= \log d_K,\qquad 
& \kappa = \frac1{\sqrt5}, \qquad \nonumber
\\ & 1< \s < 1.15, \qquad
&
\sigma_1=\frac12+\frac12\sqrt{1+4\sigma^2}, \qquad\nonumber
\\ &
s=\sigma + it,\qquad
& s_1=\sigma_1 + it.\qquad 
\end{eqnarray}
Let $\varrho_0=\beta_0 + i\gamma_0$.
Note that, by symmetry of the zeros of $\zeta_K(s)$, it suffices to consider $\g_0\ge0$.
The classical proof of the zero-free region studies the logarithmic derivative of the zeta function.
Lagarias et al. \cite{LMO} consider
$
-\Re \frac{\zeta_K'}{\zeta_K}(s)
$
and follow de La Vall\'ee Poussin's argument.
Instead, we study the differenced function
\begin{equation}\label{def-fL} 
f(\s,t) = -\Re \left( \frac{\zeta_K'}{\zeta_K}(s)-\kappa \frac{\zeta_K'}{\zeta_K}(s_{1})\right),
\end{equation}
as introduced by Stechkin. 
Observe that, for $\s>1$,
\begin{equation}\label{def-fL1} 
f(\s,t) = \sum_{\substack{\mathfrak{a} \subset  \mathcal{O}_K \\ \mathfrak{a} \ne 0}} \frac{\Lambda(\mathfrak{a})}{(\mathbb{N} \mathfrak{a})^{\s}} \left(1-\frac{\kappa }{(\mathbb{N} \mathfrak{a})^{\s_1-\s}}\right) \cos \left( t\log  \mathbb{N} \mathfrak{a} \right).
\end{equation}
\subsection{Setting up the argument}\label{Setting-up-argument}
\subsubsection{First ingredient: a trigonometric inequality. }
Let $P$ be a non-negative trigonometric polynomial of degree $n$ of the form
\[
P(\theta)  = \sum_{k=0}^n a_{k} \cos(k\theta),
\]
where all the $a_k$'s are positive.  
For example, de La Vall\'ee Poussin used
\[ 2(1+\cos \theta)^2 = 3+4\cos\theta+\cos(2\theta).\]
Later, higher degree polynomials were explored.
For example, Rosser and Schoenfeld in \cite{RS2} for $\zeta(s)$, and then McCurley in \cite{MC1} for $\L_q(s)$ used:
\begin{align*}
&P(\theta) = 8(0.9126+\cos \theta)^2(0.2766+\cos\theta)^2 = \sum_{k=0}^4 a_k \cos(k\theta),\\
\text{ with }
& a_{0} = 11.18593553 ,\  a_{1} = 19.07334401,\   a_{2} = 11.67618784,\  a_{3} = 4.7568,\  a_{4} = 1.
\end{align*}
Our choice of $P$ depends of the size of the imaginary part $\g_0$.
For example, for $\varrho_0$ real, we only consider
\[P(\theta) = 1.\]
Taking $\theta = t\log \mathbb{N} \mathfrak{a}$, we combine the trigonometric polynomial with \eqref{def-fL1} and define
\begin{equation}\label{def-S}
S(\s,\g_0) =\sum_{k=0}^n a_{k} f(\s,k\g_0)  = \sum_{\substack{\mathfrak{a} \subset  \mathcal{O}_K \\ \mathfrak{a} \ne 0}} \frac{\Lambda( \mathfrak{a})}{(\mathbb{N} \mathfrak{a})^{\s}} \left(1-\frac{\kappa }{(\mathbb{N} \mathfrak{a})^{\s_1-\s}}\right) P\left(  \g_0 \log  \mathbb{N} \mathfrak{a} \right) .
\end{equation}
Thanks to the choice of $\s_1$ and $\kappa$ as in \eqref{eq:cond}, we have
\[
1-\frac{\kappa }{(\mathbb{N} \mathfrak{a})^{\s_1-\s}} \ge 0,\ \text{ for all non-zero ideals}.
\]
Together with the non-negativity of $P$, we obtain
\begin{equation}\label{positive}
 S(\s,\g_0) \ge 0.
\end{equation}
\subsubsection{Second ingredient: an explicit formula. }
We recall the explicit formula for Dedekind Zeta functions (see equation (8.3) of \cite{LO}): 
\[
-\Re \frac{\zeta_K'}{\zeta_K}(s) = 
-\sum_{\varrho} \Re \frac1{s-\varrho}
+ \frac{1}2\log d_K
+\Re \frac1s + \Re \frac1{s-1}
+ \Re\frac{\gamma'_K}{\gamma_K}(s),
\]
where $\varrho$ runs through the non-trivial zeros of $\zeta_K$. 
It follows for $f(\s,t)$ given by \eqref{def-fL} that
 \begin{multline}\label{Stechkin-explicit-formula-0}
f(\s,t) =  
-\sum_{\varrho} \Re \left(\frac1{s-\varrho}-\kappa\frac1{s_{1}-\varrho}\right)  
+ \frac{1-\kappa}2\log d_K
\\+\Re\left(\frac1s + \frac1{s-1}-\frac{\kappa}{s_{1}} - \frac{\kappa}{s_{1}-1}\right)  + \Re\left( \frac{\gamma'_K}{\gamma_K}(s) -\kappa \frac{\gamma'_K}{\gamma_K}(s_{1}) \right).
\end{multline}
To simplify notation, we set
\begin{equation*}\label{def-F}
 F(s,z) = \Re \left( \frac1{s-z} +\frac1{s-1+\overline{z}} \right).
\end{equation*}
Using the symmetry of the zeros with respect to the $1/2$-line, we have
\begin{equation}\label{sum-zeros-0}
-\sum_{\varrho} \Re \left(\frac1{s-\varrho}-\kappa\frac1{s_{1}-\varrho}\right)
=  -\sideset{}{'} \sum_{\beta \ge \frac{1}{2}}\left(F(s,{\varrho})-\kappa F(s_1,{\varrho})\right),
 \end{equation}
where
$\ds{\sideset{}{'} \sum_{\beta \ge \frac{1}{2}} = \frac12 \sum_{\Re \varrho=1/2} + \sum_{1/2<\Re \varrho\le 1}}$. 
It follows that
\begin{multline}
\label{eq-trig1}
f(\s,k\g_0) = -\sideset{}{'} \sum_{\beta \ge \frac{1}{2}} \left(F(\s+ik\g_0,{\varrho})-\kappa F(\s_1+ik\g_0,{\varrho})\right)  
+ \frac{1-\kappa}2\log d_K
\\+F(\s+ik\g_0,1)-\kappa F(\s_1+ik\g_0,1)  
+ \Re\big( \frac{\gamma'_K}{\gamma_K}(\s+ik\g_0) -\kappa \frac{\gamma'_K}{\gamma_K}(\s_{1}+ik\g_0) \big).
\end{multline} 
Now \eqref{positive} becomes
\begin{equation}\label{positive-explicit-formula}
S_1(\s,\g_0)+S_2 +S_3(\s,\g_0)+S_4(\s,\g_0) \ge 0
\end{equation} 
where
\begin{align}
&S_{1}(\s,\g_0) =  - \sum_{k=0}^n a_k \sideset{}{'} \sum_{\beta \ge \frac{1}{2}} \left(F(\s+ik\g_0,{\varrho})-\kappa F(\s_1+ik\g_0,{\varrho})\right) , \label{def-S1}\\
&S_{2} =  \frac{1-\kappa}2(\sum_{k=0}^n a_k )\, \log d_K ,\\
&S_{3}(\s,\g_0) = \sum_{k=0}^n a_k  \left( F(\s+ik\g_0,1)-\kappa F(\s_1+ik\g_0,1) \right) ,\label{def-S3}\\
&S_{4}(\s,\g_0) =  \sum_{k=0}^n a_k \Re\left( \frac{\gamma'_K}{\gamma_K}(\s+ik\g_0) -\kappa \frac{\gamma'_K}{\gamma_K}(\s_{1}+ik\g_0) \right).\label{def-S4}
\end{align}
It remains to bound each of the $S_i$'s so as to exhibit $\b_0$ and deduce from \eqref{positive-explicit-formula} an upper  bound for it.
Observe that for $\Re z \le1$, $\Re(\s_1+ik\g_0-z)$ is large enough, implying that the terms $F(\s_1+ik\g_0,z)$ are insignificant.
On the other hand, Stechkin's trick reduces the coefficient of $\log d_K $ from $\frac12$ to $\frac{1-\kappa}2$.
We choose the parameter $\sigma$ such that $\s-1$ and $\s-\b_0$ are both of size $\frac1{\L}$.
We split our argument and evaluate each of the $S_i$'s for:
\begin{itemize}
 \item[Case 1:] $\g_0>1$,
 \item[Case 2:] $\frac{d_2}{\L} < \g_0 \le 1$,
 \item[Case 3:] $\frac{d_1}{\L}<\g_0 \le \frac{d_2}{\L} $,
 \item[Case 4:] $0 < \g_0\le \frac{d_1}{\L}$,
 \item[Case 5:] $\gamma_0=0$.
In this case, we consider $\beta_1$ and $\beta_2$ two real zeros satisfying $\b_1 \le \b_2$.
It is possible to prove an upper bound for $\b_1$, and thus establish a region free of zeros, with the exception of $\b_2$.
\end{itemize}
Here $d_1$ and $d_2$ are positive constants chosen to make the zero-free regions as wide as possible.
For each case, we make a specific choice for the trigonometric coefficients $a_k$.
\subsubsection{Bounding the sum over the zeros $S_1$. }
In de La Vall\'ee Poussin's argument, he makes use of the positivity condition 
\[\Re \frac1{\s+ik\g_0-\varrho} \ge 0, \ \text{for all non-trivial zeros } \varrho \text{ and for all }   \s>1.\]
Later, Stechkin showed (see Lemma \ref{Stechkin} below) that 
\[F(\s+ik\g_0,{\varrho})-\kappa F(\s_1+ik\g_0,{\varrho}) \ge 0, \ \text{for all non-trivial zeros}\ \varrho.\]
Moreover, $\kappa=1/\sqrt5$ is the largest value such that the inequality holds.
Observe that for the zeros $\varrho$ where $|\s+ik\g_0-\varrho|$ is small, then $\Re \frac1{\s+ik\g_0-\varrho} $ is a large positive term.
We retain these zeros in the sum \eqref{def-S1} and we discard the other ones by using Stechkin's Lemma.
\\
Cases 1, 2, and 3: We have $\g_0 \gg \s-\b_0$. We isolate $\varrho_0=\beta_0+i\gamma_0$ only for the $k=1$ term:
\[
S_1(\s,\g_0) \le -\frac{a_1}{\s-\beta_0} + \mathcal{O}(1).
\]
Case 4:
We isolate both zeros $\varrho_0$ and $\overline{\varrho_0}$:
\[
S_1(\s,0) \le - \frac{2(1+o(1))\sum_{k=0}^n a_k}{\s-\beta_0} + \mathcal{O}(1).
\]
Case 5: 
We isolate both zeros $\b_1$ and $\b_2$:
\[
S_1(\s,0) \le - \frac{2(1+o(1))}{\s-\beta_1} + \mathcal{O}(1).
\]
\subsubsection{Bounding the polar terms $S_3$. }
Note that $F(\s+ik\g_0,1)-\kappa F(\s_1+ik\g_0,1)$ is essentially of size $\Re \frac{1}{\s-1+ik\g_0}$. 
\\
Cases 1 and 2:
We have $\g_0 \gg\s-1$. Thus $F(\s+ik\g_0,1)-\kappa F(\s_1+ik\g_0,1) \ll 1$ for all $k\not=0$ and
\[
S_3(\s,\g_0)  \le  \frac{a_0}{\s-1} + \mathcal{O}(1).
\]
Cases 3, 4, and 5:
Since $\g_0 \ll \s-1$, $F(\s+ik\g_0,1)-\kappa F(\s_1+ik\g_0,1) \ll  \frac{1}{\s-1}$ for all $k$ and
\[
S_3(\s,\g_0) \le \frac{(1+o(1)) \sum_{k=0}^n a_k}{\s-1} + \mathcal{O}(1).
\]
\subsubsection{Bounding the $\gamma_K$ terms $S_4$. }
An analysis of $\psi(x)=\frac{\Gamma'}{\Gamma}(x)$ together with the definition \eqref{eq:gamK} of $\frac{\gamma_K'}{\gamma_K}$ gives 
\[
\sum_{k=0}^n a_k \Re\frac{\gamma'_K}{\gamma_K}(\s+ik\g_0) 
\le
\begin{cases}
\frac{(1+o(1))(1-\kappa)}{2}(\log \gamma_0) n_K \text{  in Case 1},\\
\mathcal{O} (n_K)  \text{  in Case } 2, 3, 4, 5.
\end{cases}
\]
\subsubsection{Conclusion.} 
We deduce from the above bounds an inequality depending on $\beta_0$ (respectively $\b_1$), $\g_0, d_K, n_K$, and $\s$.
We choose $\s$ so as to obtain the smallest upper bound possible for $\b_0$ and $\b_1$.

The following sections establish in complete detail the results mentioned in section \ref{Setting-up-argument}.
\subsection{Preliminary Lemma about the zero terms.}
We define $s_1(\s,\g_0,k)$ to be the $k$-th summand of $S_1(\s,\g_0)$:
\[
s_1(\s,\g_0,k)=
-\sideset{}{'} \sum_{\beta \ge \frac{1}{2}} \left(F(\s+ik\g_0,\varrho)-\kappa F(\s_1+ik\g_0,{\varrho})\right).
\]
We employ the following lemma to establish a bound for it.
\begin{lemma}[Stechkin - \cite{St1}]\label{Stechkin}
Let $s=\s+it$ with $\sigma>1$.
If $0<\Re z  <1$, then
\begin{equation}\label{eq-S1}
F(s,z)-\kappa F(s_1,z) \ge 0. 
\end{equation}
If $\Im z = t$ and $1/2 \le \Re z <1$, then
\begin{equation}\label{eq-S2}
 \Re \frac1{s-1+\overline{z}}-\kappa F(s_1,z) \ge 0. 
\end{equation}
\end{lemma}
For the rest of the article, we consider $\varrho_0=\b_0+i\g_0$ a non-trivial zero of $\zeta_K$. We assume
\begin{equation}\label{eq:cond-rho0}
\b_0 \ge 0.85\  \text{ and }\ \gamma_0\ge 0. 
\end{equation}
Note that otherwise, the zero-free region $\Re s > 0.85$ is established. 
When $k=1$, we isolate $\varrho_0$ from the sum in \myref{sum-zeros-0}. Together with \myref{eq-S1} and  \myref{eq-S2}, we obtain 
\begin{multline*}
s_1(\s,\g_0,1)
\le - \left(F(\s+i\g_0,\varrho_0)-\kappa F(\s_1+i\g_0,\varrho_0)\right)
\\=- \left(\Re\frac1{\s+i\g_0-\varrho_0}+\Re\frac1{\s+i\g_0-1+\overline{\varrho_0}}-\kappa F(\s_1+i\g_0,\varrho_0)\right)
\le - \frac1{\s-\b_0}.
\end{multline*}
When $k=0,2,3,4$, we consider various cases.\\
If $ \g_0>1$ (Case 1), we use \myref{eq-S1} for all zeros:
\[
s_1(\s,\g_0,k) \le 0.
\]
For $\g_0$  as in Case 2, we apply \myref{eq-S1} except for $\varrho_0$:
\begin{equation}\label{eq-isolate2}
s_1(\s,\g_0,k)  \le - \left(F(\s+ik\g_0,\b_0+i\g_0)-\kappa F(\s_1+ik\g_0,\b_0+i\g_0)\right).
\end{equation}
For $\g_0$  as in Case 3, we apply \myref{eq-S1} except for $\varrho_0$ and $\overline{\varrho_0}$:
\begin{multline}\label{eq-isolate31}
s_1(\s,\g_0,k) \le - \left(F(\s+i\g_0,\b_0+i\g_0)-\kappa F(\s_1+i\g_0,\b_0+i\g_0)\right)
\\-\left(F(\s+i\g_0,\b_0-i\g_0)-\kappa F(\s_1+i\g_0,\b_0-i\g_0)\right)  .
\end{multline}
We observe that, for $x$ real, $F(x,\varrho_0) = F(x,\overline{\varrho_0})$.
For Cases 4 and 5, \myref{eq-isolate31} becomes
\begin{equation}\label{eq-isolate32}
s_1(\s,0,k) \le - 2\left(F(\s,\b_0+i\g_0)-\kappa F(\s_1,\b_0+i\g_0)\right).
\end{equation}
It remains to bound $-\left( F(\s+ik\g_0,\b_0\pm i\g_0)- \kappa F(\s_1+ik\g_0,\b_0 \pm i\g_0)\right)$ in \myref{eq-isolate2}, \myref{eq-isolate31}, and \myref{eq-isolate32} for $k=0,1,2,3,4$ and $\g_0\le 1$.
The following elementary lemma may be used for this.
\begin{lemma} \label{gbds}
For $a, b, c >0$, we define 
\begin{equation*}
  \label{eq:g}
  g(a,b,c;x) = \kappa \Big( 
  \frac{a}{a^2+x^2} + \frac{b}{b^2+x^2}
  \Big) - \frac{c}{c^2+x^2}.
\end{equation*}
Let $a_0 = \frac{\sqrt{5}-1}{2}, b_0 = \frac{1+\sqrt{5}}{2}$, and $c_0 =1$.  \\
(i) Let $g_{0} = -0.121585107$. Then the inequality 
\[
  g_{0}  \le  g(a_0,b_0,c_0;x) \le 0  
\]
is valid for all $x \in \mathbb{R}$. \\
(ii) Let $a,b,c >0$ and let $0 < \epsilon < \epsilon_0$.  
If there exist constants $m_1,m_2,m_3$ such that $|a-a_0| < m_1 \epsilon $, $|b-b_0| <  m_2 \epsilon$, and $|c-c_0| <  m_3 \epsilon$,
then 
\[
    g_{0} -m_0 \epsilon \le  g(a,b,c;x) \le m_0 \epsilon, 
\]
where \[m_0=
\frac{ \kappa m_1}{(a_0 -m_1 \epsilon_0)^2}
+ \frac{ \kappa m_2}{(b_0 -m_2 \epsilon_0)^2}
 + \frac{m_3}{(c_0 - m_3 \epsilon_0)^2} 
.\]
\end{lemma}
\begin{proof} (i) Differentiating we find that 
\begin{align*}
   g(a_0,b_0,c_0;x)  &= 2x \Big(  \kappa \Big(- \frac{a_0}{(a_{0}^2+x^2)^2} +\frac{b_0}{(b_{0}^2+x^2)^2}  \Big) + \frac{c_0}{(c_{0}^2+x^2)^2}  \Big) \\
    & =  \frac{2x (2x^6 +4x^4-1)}{(a_{0}^2+x^2)^2(b_{0}^2+x^2)^2(c_{0}^2+x^2)^2}.
\end{align*}
The polynomial $2x^6+4x^4-1$ has one positive real root  $x_0 = 0.672016341 \ldots $.  It follows that 
$g(a_0,b_0,c_0;x)$ is decreasing on $(0,x_0)$ and increasing on $(x_0, \infty)$. 
Using these facts, the first part of the lemma follows as $g(a_0,b_0,c_0;0) = 0$, $g(a_0,b_0,c_0; x_0) = g_{0}= -0.121585107 \ldots$,
and $\lim_{x \to \infty} g(a_0,b_0,c_0;x) =0$.  \\
(ii) We begin by considering the difference 
\begin{multline*}
   g(a,b,c;x) - g(a_0,b_0,c_0;x) = 
\\\kappa \Big( 
  \frac{a}{a^2+x^2}-\frac{a_0}{a_{0}^2+x^2} + \frac{b}{b^2+x^2}-\frac{b_0}{b_{0}^2+x^2}
  \Big) - \frac{c}{c^2+x^2} + \frac{c_0}{c_{0}^2+x^2}.
\end{multline*}
For positive real numbers $u$ and $u_0$ we have that 
\begin{multline*}
   \Big| \frac{u}{u^2+x^2} - \frac{u_0}{u_{0}^2 +x^2} \Big|
  =  \Big| \frac{(u-u_0)(x^2-uu_0)}{(u^2+x^2)(u_{0}^2+x^2)} \Big|
\\ \le \frac{|u-u_0|(x^2+\max(u,u_0)^2)}{(x^2+u^2)(x^2+u_{0}^2)} 
\le \frac{|u-u_0|}{\min(u,u_0)^2}. 
\end{multline*}
Using this bound, the triangle inequality implies that  
\begin{multline*}
  |g(a,b,c;x) - g(a_0,b_0,c_0;x)|  
\\
 \le \epsilon \Big( 
   \frac{ \kappa m_1}{(a_0 -m_1 \epsilon_0)^2}
+ \frac{ \kappa m_2}{(b_0 -m_2 \epsilon_0)^2}
 + \frac{m_3}{(c_0 - m_3 \epsilon_0)^2} 
\Big)   =m_0 \epsilon.
\end{multline*}
Hence $g_{0} - m_0 \epsilon \le g(a,b,c;x) \le m_0 \epsilon$. 
\end{proof}
Observe that
\begin{align*}
& -\left( F(\s+ik\g_0,\b_0+i\g_0)- \kappa F(\s_1+ik\g_0,\b_0+i\g_0)\right) 
\\&= - \frac{\s-\b_0}{(\s-\b_0)^2+((k-1)\g_0)^2}- \frac{\s-1+\b_0}{(\s-1+\b_0)^2+((k-1)\g_0)^2} 
\\&+  \kappa\frac{\s_1-\b_0}{(\s_1-\b_0)^2+((k-1)\g_0)^2} + \kappa \frac{\s_1-1+\b_0}{(\s_1-1+\b_0)^2+((k-1)\g_0)^2}  
\\&= - \frac{\s-\b_0}{(\s-\b_0)^2+((k-1)\g_0)^2} - g(a,b,c;x),
\end{align*}
where $a = \sigma_1-\beta$, $b = \sigma_1-1+\beta$, $c=\sigma-1+\beta$, and $x =|k-1|\g_0$. 
From \myref{eq:cond-rho0}, it follows that $1-\epsilon \le \beta < 1$ and $1 < \sigma \le 1+\epsilon$, with $\epsilon=0.15$.  
Thus
\[
  |a-a_0| \le 1.9064\epsilon, |b-b_0| \le 0.9064 \epsilon, \text{ and } |c-c_0| \le \epsilon.
\]
We find $m_0=9.3001$. Thus $-g(a,b,c;x)\le m_0\epsilon = 1.3951$.
Collecting the inequalities proven above in this section, we obtain the following:
\begin{lemma}\label{Stechkin1}
Assume \myref{eq:cond} and \myref{eq:cond-rho0}.
For $k=1$ in Cases 1 to 3, we have 
\begin{equation}\label{eq-isolate4}
-\sideset{}{'} \sum_{\beta \ge \frac{1}{2}} \left(F(\s+i\g_0,\varrho)-\kappa F(\s_1+i\g_0,{\varrho})\right)
\le - \frac1{\sigma-\beta_0}.
\end{equation}
For $k\not=1$ and $\g_0>1$ as in Case 1, we have 
\begin{equation}\label{eq-isolate5}
-\sideset{}{'} \sum_{\beta \ge \frac{1}{2}} \left(F(\s+ik\g_0,\varrho)-\kappa F(\s_1+ik\g_0,{\varrho})\right) \le 0 .
\end{equation}
Let $\alpha_1 =1.3951$.
For $k\not=1$ and $\g_0\le 1$ as in Case 2, we have 
\begin{equation}\label{eq-isolate6}
-\sideset{}{'} \sum_{\beta \ge \frac{1}{2}} \left(F(\s+ik\g_0,\varrho)-\kappa F(\s_1+ik\g_0,{\varrho})\right)
\le - \frac{\sigma-\beta_0}{(\sigma-\beta_0)^2+(k-1)^2\g_0^2} + \alpha_1 .
\end{equation}
For $k\not=1$ and $\g_0\le 1$ as in Case 3, we have 
\begin{multline}\label{eq-isolate7}
-\sideset{}{'} \sum_{\beta \ge \frac{1}{2}} \left(F(\s+ik\g_0,\varrho)-\kappa F(\s_1+ik\g_0,{\varrho})\right)
\\ \le - \frac{\sigma-\beta_0}{(\sigma-\beta_0)^2+(k-1)^2\g_0^2} - \frac{\sigma-\beta_0}{(\sigma-\beta_0)^2+(k+1)^2\g_0^2} + 2\alpha_1.
\end{multline}
Moreover, for $\g_0\le 1$ as in Cases 4 and 5, we have 
\begin{equation}\label{eq-isolate8}
-\sideset{}{'} \sum_{\beta \ge \frac{1}{2}} \left(F(\s,\varrho)-\kappa F(\s_1,{\varrho})\right)
 \le - 2 \frac{\sigma-\beta_0}{(\sigma-\beta_0)^2+\g_0^2} + 2\alpha_1. 
\end{equation} 
\end{lemma}
\subsection{Preliminary Lemma about the polar terms}
%
\begin{lemma}\label{lem-pole}
Assume \myref{eq:cond}. 
We define 
\[\alpha_{20}=0.0215
\,,\quad
\alpha_{21}= 1.5166
\,,\quad
\alpha_{22}= 1.6666.\]
If $k=0$ or $\g_0=0$, then
\begin{equation}\label{pole0}
F(\s+ik\g_0,1)-\kappa F(\s_1+ik\g_0,1) 
\le \frac1{\s-1} + \alpha_{20}. 
\end{equation}
If $k=1,2,3,4$ and $0<\g_0<1$, then
\begin{equation}\label{pole1}
F(\s+ik\g_0,1)-\kappa F(\s_1+ik\g_0,1) 
\le \frac{\s-1}{(\s-1)^2+\left(k\g_0\right)^2} + \alpha_{21}. 
\end{equation}
If $k=1,2,3,4$ and $\g_0\ge1$, then
\begin{equation}\label{pole2}
F(\s+ik\g_0,1)-\kappa F(\s_1+ik\g_0,1) 
\le  \alpha_{22}.
\end{equation}
\end{lemma}
\begin{proof}
When $k\g_0=0$, we have
\[F(\s,1)-\kappa F(\s_1,1) = \frac1{\s-1} + \frac1{\s}-\frac{\kappa}{\s_1}-\frac{\kappa}{\s_1-1}.\]
A Maple computation shows that the maximum of $\frac1{\s}-\frac{\kappa}{\s_1}-\frac{\kappa}{\s_1-1}$ for $\s  \in[1,1.15]$ occurs at $\s=1.15$, and is $0.02146\ldots$. 
Thus
\begin{equation}\label{c1}
F(\s,1)-\kappa F(\s_1,1) \le \frac1{\s-1} + \alpha_{20}.
\end{equation}
Observe that
\begin{multline*}
F(\s+ik\g_0,1)-\kappa F(\s_1+ik\g_0,1)
\\ =  \frac{\s}{\s^2+(k\g_0)^2} + \frac{\s-1}{(\s-1)^2+(k\g_0)^2}-\frac{\kappa \s_{1}}{\s_{1}^2+(k\g_0)^2} - \frac{ \kappa(\s_{1}-1)}{(\s_{1}-1)^2+(k\g_0)^2}
\\ = \frac{\s-1}{(\s-1)^2+(k\g_0)^2} - g(\s_1-1,\s_1,\s;k\g_0).
\end{multline*}
Taking $\epsilon=0.15$, it follows from Lemma \ref{gbds} that
\begin{equation*}
g(\s_1-1,\s_1,\s;k\g_0) \ge g_0- m_0\epsilon \ge -1.5166.
\end{equation*}
Thus
\begin{equation}\label{c2}
F(\s+ik\g_0,1)-\kappa F(\s_1+ik\g_0,1)
\le  \frac{\s-1}{(\s-1)^2+(k\g_0)^2} + \alpha_{21}.
\end{equation}
Moreover, when $\g_0\ge1$, the above becomes
\begin{equation}\label{c3}
F(\s+ik\g_0,1)-\kappa F(\s_1+ik\g_0,1)
\le 0.15 + \alpha_{21}.
\end{equation}
The announced inequalities follow from \myref{c1}, \myref{c2}, and \myref{c3}.
\end{proof}
\subsection{Preliminary Lemma about the $\gamma_K$ terms}
We now bound the expression 
\begin{multline*}
\Re\left( \frac{\gamma'_K}{\gamma_K}(\s+ik\g_0) -\kappa \frac{\gamma'_K}{\gamma_K}(\s_{1}+ik\g_0) \right)
 = -\frac{(1- \kappa)\log \pi}{2}  n_K
\\+ \frac{a}2 \Re \left( \frac{\Gamma'}{\Gamma}\left(\tfrac{\s+ik\g_0}{2}\right) 
 -\kappa  \frac{\Gamma'}{\Gamma}\left(\tfrac{\s_{1}+ik\g_0}{2}\right) \right)
+  \frac{b}2 \Re \left( \frac{\Gamma'}{\Gamma}\left(\tfrac{\s+ik\g_0+1}{2}\right) -\kappa \frac{\Gamma'}{\Gamma}\left(\tfrac{\s_{1}+ik\g_0+1}{2}\right)\right),
\end{multline*}
where $a+b=n_K$ as in \eqref{eq:gamK}.
We have: 
\begin{lemma}
\label{boundGamma<1}
Assume \myref{eq:cond}. Let $k=0,1,2,3,4$ and $c=0$ or $1$. Then
\begin{multline*}
\frac12 \Re\left( \frac{\Gamma'}{\Gamma}\left(\tfrac{\s+ik\g_0+c}2\right) - \kappa \frac{\Gamma'}{\Gamma}\left(\tfrac{\s_{1}+ik\g_0+c}2\right) \right)
\\ \le 
\begin{cases}
d(0)=D(0) & \hbox{ if } k \g_0 = 0,\\
d(k) & \hbox{ if } 0 < \g_0 \le 1 \hbox{ and }  k=1,2,3,4,\\
\tfrac{1-\kappa}2 \log \g_0 + D(k) & \hbox{ if } \g_0 > 1 \hbox{ and }  k=1,2,3,4,
\end{cases} 
\end{multline*}
where admissible values of $D(k)$ and $d(k)$ are given in the following chart.
 \[
 \begin{array}{|c|c|c|c|c|c|}
 \hline
 k & 0 & 1 & 2 & 3 & 4 \\
 \hline
 d(k) & -0.0512 & -0.0390 & 0.2469 & 0.4452 & 0.5842 \\
 \hline
 D(k) & -0.0512 & 0.3918 & 0.3915 & 0.4062 & 0.4266 \\
 \hline
 \end{array}
 \]
\end{lemma}
The cases $k=1,2,3,4$ are a direct consequence of Lemmas 1 and 2 of \cite{MC1}. 
The case $k=0$ can be obtained by a Maple computation.
Thus we deduce
\begin{lemma} \label{boundGammaL}
Assume \myref{eq:cond}.
If $k$ or $\g_0 = 0$, then
\begin{equation}\label{gamma0}
\Re\left( \frac{\gamma'_K}{\gamma_K}(\s) -\kappa \frac{\gamma'_K}{\gamma_K}(\s_{1}) \right)
\le 
\left( -\frac{(1- \kappa)\log \pi}{2}  + d(0) \right) n_K. 
\end{equation}
If $0 < \g_0 \le 1$ and $k=1,2,3,4$, then
\begin{equation}\label{gamma1}
\Re\left( \frac{\gamma'_K}{\gamma_K}(\s+ik\g_0) -\kappa \frac{\gamma'_K}{\gamma_K}(\s_{1}+ik\g_0) \right)
\le 
\left( -\frac{(1- \kappa)\log \pi}{2}   + d(k) \right) n_K .
\end{equation}
If $\g_0 > 1$ and $k=1,2,3,4$, then
\begin{equation}\label{gamma2}
\Re\left( \frac{\gamma'_K}{\gamma_K}(\s+ik\g_0) -\kappa \frac{\gamma'_K}{\gamma_K}(\s_{1}+ik\g_0) \right)
\le 
\left(  \frac{1-\kappa}2  \log \g_0 + \frac{1-\kappa}2 \log \frac{k}{\pi} + D(k) \right) n_K.
\end{equation}
\end{lemma}
\section{Zero-free regions}
In this section, we continue to assume conditions \myref{eq:cond} and \myref{eq:cond-rho0} for the parameters $\s,\L$, and for the non-trivial zero $\varrho_0=\b_0+i\g_0$.
We use lemmas \ref{Stechkin1}, \ref{lem-pole}, and \ref{boundGammaL} from Section \ref{preliminary-lemmas} to provide upper bounds for the $S_j$'s and thus derive zero-free regions.

In order to simplify future computations, we record the following elementary lemma:
\begin{lemma}\label{elementary}
Let $a,b,q,t>0$ be fixed. 
\begin{enumerate}
\item[(i)] If $2a-b>0$, then 
$\ds{f_1(x) = - \frac{a}{x}- \frac{bx}{x^2+t^2}}$
is increasing.
\item[(ii)] $\ds{ f_2(a,b;x)= \frac{a}{a^2+x^2} - \frac{b}{b^2+x^2} }$
has opposite sign of $(b-a)(x^2-ab)$.
\item[(iii)] If $qb^3 \ge a^3, qb\ge a$, and $qa\ge b$, then $\ds{ f_3(a,b,q;x)= q\frac{a}{a^2+x^2} - \frac{b}{b^2+x^2} }$
is decreasing with $x$.
\end{enumerate}
\end{lemma}
\begin{proof}
We have 
\begin{align*}
& f_1'(x)  = \frac{(a+b)x^4+(2a-b)t^2x^2+at^4}{x^2(x^2+t^2)^2} , \\
& f_2(a,b;x) = -\frac{(b-a)(x^2-ab)}{(a^2+x^2)(b^2+x^2)} , \\
& \frac{\partial}{\partial x} f_3(a,b,q;x)
= (-2x)\left( \frac{ab(qb^3-a^3) + 2ab(qb-a)x^2+(qa-b)x^4)}{(a^2+x^2)^2(b^2+x^2)^2}\right).
\end{align*}
The lemma follows from the above three formulae.
\end{proof}
\subsection{Case 1: Zero-free region when $\g_0>1$}\label{large-region}
Let $r>0$. We choose $\s$ such that \[\s-1=r(1-\b_0).\]
We define the trigonometric polynomial
\begin{align*}
&P(\theta) = 8( 0.8924+\cos \theta)^2( 0.1768+\cos\theta)^2= \sum_{k=0}^4 a_k \cos(k\theta),
\\
& a_{0} = 9.034112058 ,\ a_{1} = 15.52951106, \ a_{2} = 9.834965120 ,\ a_{3} = 4.2768, \ a_{4} = 1.  
\end{align*}
We apply Lemma \ref{Stechkin1}, using equations \myref{eq-isolate4} when $k=1$, and \myref{eq-isolate5} when $k=0,2,3,4$:
\[
S_{1}(\s,\g_0)
\le - \frac{a_1}{\sigma-\beta_0} .
\]  
We apply Lemma \ref{lem-pole}, using equations \myref{pole0} when $k=0$, and \myref{pole2} otherwise:
\[
S_{3}(\s,\g_0)
\le \frac{a_0}{\s-1} + a_0\alpha_{20}
+ \alpha_{22} \sum_{k=1}^4 a_k. 
\]
We apply Lemma \ref{boundGamma<1}, using equations \myref{gamma0} for $k=0$ and \myref{gamma2} otherwise:
\begin{multline*}
S_{4}(\s,\g_0)
\le 
a_0 \Big( - \frac{(1- \kappa)\log \pi}2  + d(0) \Big) n_K 
\\+ 
\sum_{k=1}^4 a_k \Big(  \frac{1-\kappa}2  \log \g_0 + \frac{1-\kappa}2 \log \frac{k}{\pi} + D(k) \Big) n_K .
\end{multline*}
Together with \myref{positive-explicit-formula}, \myref{eq-trig1}, and the above inequalities, we deduce
\[
0 \le \Big(- \frac{a_1}{1+r} 
+ \frac{a_0}{r} \Big)\frac1{1-\b_0} 
+ c_1 \log d_K + c_2 (\log\g_0) n_K  + c_3 n_K + c_4, 
\]
where 
\begin{align*}
& c_1 = \frac{1-\kappa}2 \sum_{k=0}^4 a_k,\qquad
 c_2= \frac{1-\kappa}2 \Big(\sum_{k=1}^4 a_k\Big) ,\\ 
& c_3= a_0 \Big(  - \frac{(1- \kappa)\log \pi}2  + d(0) \Big) 
+  \sum_{k=1}^4 a_k \Big( \frac{1-\kappa}2 \log \frac{k}{\pi} + D(k) \Big) ,\\ 
& c_4=  a_0\alpha_{20}
+ \alpha_{22} \sum_{k=1}^4 a_k.
\end{align*}
Thus
\[
\b_0 \le 1- \frac{\frac{a_1}{1+r}-\frac{a_0}{r}}{c_1 \log d_K + c_2  (\log\g_0) n_K  + c_3 n_K + c_4}.
\]
The largest value for $\frac{a_1}{1+r}-\frac{a_0}{r}$ occurs for $r=\frac{\sqrt{a_0}}{\sqrt{a_1}-\sqrt{a_0}}=3.21438\ldots $
and hence
\[
\b_0 \le 1- \frac{1}{     
12.5419 \log d_K +  
9.6861 (\log\g_0) n_K + 
3.0297 n_K + 
58.6265}.
\]
This proves the zero free region \myref{zfrD1} of Theorem \ref{ZFR}.
\begin{remark}\
We ran a Maple computation to determine the $a_k$'s which minimized
\[
\frac{c_1}{\frac{a_1}{1+r}-\frac{a_0}{r}}
= \frac{\frac{1-\kappa}2 \sum_{k=0}^4 a_k}{(\sqrt{a_1}- \sqrt{a_0})^2}.
\]
\end{remark}
\noindent
Let $0<r,c<1$. For the remainder of the article, we consider
\[
\s-1 = \frac{r}{\L}\ \text{and} \ 
1-\b_0= \frac{c}{\L}.
\]
\subsection{Case 2: Zero-free region when $\frac{d_2}{\L}<\gamma_0\le1$, with $d_2=2.374$.}\label{superior-region}
For the trigonometric polynomial, we choose
\begin{align*}
&P(\theta) = 8(0.8918+\cos \theta)^2(0.1732+\cos\theta)^2= \sum_{k=0}^4 a_k \cos(k\theta), \\
&a_{0} = 8.96344062,\ a_{1} = 15.41199431, \ a_{2} = 9.77257808,\ a_{3} = 4.26, \ a_{4} = 1. 
\end{align*}
In addition to our conditions on $\s$ and $\varrho_0$, we impose the conditions
\begin{eqnarray}
& 0< \frac{a_0}{a_1-a_0}c < r <1 , \label{cond11}\\
& d_2>  \frac{\sqrt{r(r+c)}}{2} \label{cond12}.
\end{eqnarray}
We apply Lemma \ref{Stechkin1}, using \myref{eq-isolate4} for $k=1$, and \myref{eq-isolate6} for $k=k=0,2,3,4$:
\[
S_{1}(\s,\g_0)
\le - \frac{a_1}{\sigma-\beta_0}
- \sum_{k=0,2,3,4} \frac{a_k(\s-\b_0)}{(\s-\b_0)^2+(k-1)^2\g_0^2}
+ \alpha_1 \sum_{k=0,2,3,4} a_k. 
\]
We apply Lemma \ref{lem-pole}, using \myref{pole0} for $k=0$, and \myref{pole1} otherwise:
\[
S_{3}(\s,\g_0)
\le \frac{a_0}{\s-1} + a_0\alpha_{20}
+ \sum_{k=1}^4 \frac{a_k (\s-1)}{(\s-1)^2+k^2\g_0^2} 
+ \alpha_{21} \sum_{k=1}^4 a_k. 
\]
We apply Lemma \ref{boundGamma<1}, using equations \myref{gamma0} for $k=0$, and \myref{gamma1} otherwise:
\[
S_{4}(\s,\g_0)
\le 
\sum_{k=0}^4 a_k \Big(  -\frac{1-\kappa}2 \log \pi + d(k) \Big) n_K .
\]
Note that the coefficient of $n_K$ is negative and may be dispensed.
Together with \myref{positive-explicit-formula}, \myref{eq-trig1}, and the above inequalities, we deduce
 \begin{equation}
 \begin{split}\label{trig2}
0 \le &
  \frac{a_0}{\s-1}- \frac{a_1}{\sigma-\beta_0}
+   \frac{a_1(\s-1)}{(\s-1)^2+\g_0^2}
- \frac{a_0(\s-\b_0)}{(\s-\b_0)^2+\g_0^2}
+  \frac{1-\kappa}2 \log d_K  \sum_{k=0}^4 a_k
\\&+ \sum_{k=2,3,4} a_k \Big(\frac{\s-1}{(\s-1)^2+k^2\g_0^2}
-\frac{\s-\b_0}{(\s-\b_0)^2+(k-1)^2\g_0^2} \Big)
\\& + \alpha_1 \sum_{k=0,2,3,4} a_k
 + \alpha_{20} a_0
+ \alpha_{21} \sum_{k=1}^4 a_k.     
 \end{split}
 \end{equation}
The term in the second row may be dropped since, for $k=2,3,4$,
\begin{equation}\label{bound-last-zeros}
 \frac{\s-1}{(\s-1)^2+k^2\g_0^2} -\frac{\s-\b_0}{(\s-\b_0)^2+(k-1)^2\g_0^2}
\le 0.
\end{equation}
This is established as follows:
\[
\frac{\s-1}{(\s-1)^2+k^2\g_0^2} -\frac{\s-\b_0}{(\s-\b_0)^2+(k-1)^2\g_0^2},
\le f_2(a,b;x)
\] 
where $a=\s-1,b=\s-\b_0$, and $x=k \g_0$.
We have $b-a=(\s-\b_0)-(\s-1) \ge 0$ and $x^2-ab=(k \g_0)^2 -(\s-\b_0)(\s-1) \ge \frac{2d_2-rc}{\L^2} \ge 0$ by condition \myref{cond12}.
Hence Lemma \ref{elementary} gives that $f_2(a,b;x) \le 0$, and \myref{bound-last-zeros} is established.
Next, we have
\begin{equation}\label{bound-first-zeros}
 \frac{a_1(\s-1)}{(\s-1)^2+\g_0^2} -  \frac{a_0(\sigma-\beta_0)}{(\sigma-\beta_0)^2+\g_0^2}
\le a_0f_3(a,b,q;x)
\end{equation}
where $a=\s-1=\frac{r}{\L}$, $b=\s-\b_0=\frac{r+c}{\L}$, $q=\frac{a_1}{a_0}=1.71942\ldots$, and $x=\g_0$. 
The conditions of Lemma \ref{elementary} (iii) are satisfied. Hence 
$f_3(a,b,q;x)$ decreases with $x$ on $(d_2/\L,1)$ and
\[
f_3(a,b,q;x)\le f_3(a,b,q;d_2/\L) = 
\Big( q \frac{r}{r^2+d_2^2} - \frac{r+c}{(r+c)^2+d_2^2} \Big)\L.
\]
This proves \myref{bound-first-zeros}.
Together with \myref{trig2} and \myref{bound-first-zeros}, we obtain
\begin{multline*}\label{trig21}
  0 \le 
  \frac{a_0}{r} -  \frac{a_1}{r+c}
+ \frac{a_1 r}{r^2+d_2^2} 
- \frac{a_0(r+c)}{(r+c)^2+d_2^2}  
+  \frac{1-\kappa}2\sum_{k=0}^4 a_k
\\
+ \Big(\alpha_1 \sum_{k=0,2,3,4} a_k
 + a_0\alpha_{20}
+ \alpha_{21} \sum_{k=1}^4 a_k \Big)\frac1{\L},
\end{multline*}
which becomes, for $\L$ asymptotically large,
$  0 \le \mathcal{E}(d_2,r,c)$, where
\begin{equation}\label{trig22}
  \mathcal{E}(d_2,r,c)
  = \frac{a_0}{r} -  \frac{a_1}{r+c}
+ \frac{a_1 r}{r^2+d_2^2} 
- \frac{a_0(r+c)}{(r+c)^2+d_2^2}  
+  \frac{1-\kappa}2\sum_{k=0}^4 a_k.
\end{equation}
Observe that since $2a_1>a_0$, (1) of Lemma \ref{elementary} implies that $ \mathcal{E}(d_2,r,c)$ increases with $c$.
Thus the smallest value for $c=c(d_2,r)$ satisfying the above inequality is the root of $ \mathcal{E}(d_2,r,c)$.
We now choose the parameters $d_2$ and $r$ such that $c(d_2,r)$ is as small as possible.
A GP-Pari computation gives
\begin{center}
\begin{tabular}{|c|c|c|}
\hline
$d_2$ & $r$ & $1/c$ \\ 
\hline
$2.374$ & $0.248$ &  $12.7305 $ \\ 
\hline
\end{tabular}\end{center}
\begin{remark}\label{remark1}
We now give some motivation for the choice of the trigonometric polynomial.
Numerically, we expect $d_2$ to be close to $2.5$. Thus it will be much larger than the expected values for $r$ and $c$. 
To simplify the analysis of \myref{trig22}, we drop the terms depending on $d_2$. 
We expect the values for $r$ and $c$ to be very close to $\tilde{r}$ and $\tilde{c}$ respectively, where $\tilde{r}$ and $\tilde{c}$ are numbers which satisfy
\[
a_0 \frac1{\tilde{r}}  - a_1 \frac1{\tilde{r}+\tilde{c}} 
 + \frac{1-\kappa}2 \sum_{k=0}^4 a_k \ge 0.
\]
This occurs as long as  
\[\tilde{c} \ge \frac{(a_1-a_0)\tilde{r}- \left(\frac{1-\kappa}2 \sum_{k=0}^4 a_k\right)  \tilde{r}^2 }{a_0+ \left(\frac{1-\kappa}2 \sum_{k=0}^4 a_k\right)  \tilde{r}}.\]
By calculus, the expression on the right is minimized for
\[
\tilde{r} = \frac{a_0\left(\sqrt{\frac{a_1}{a_0}}-1\right)}{\frac{1-\kappa}2 \sum_{k=0}^4 a_k}. 
\]
We set $d_2=2.5$, and run a Maple computation to determine which $a_k$'s make the root $c$ of $\mathcal{E}(2.5,\tilde{r},c)=0$ as small as possible. 
\end{remark}
\subsection{Case 3: Zero-free region when $\frac{d_1}{\L}<\g_0\le\frac{d_2}{\L}$, with $d_1=1.021,\ d_2=2.374$.}\label{middle-region}
We choose for our trigonometrical polynomial:
\begin{align*}
& P(\theta) = 8(0.8924+\cos \theta)^2(0.1771+\cos\theta)^2= \sum_{k=0}^4 a_k \cos(k\theta), \\
& a_{0} = 9.039496 , \ 
a_{1} = 15.538449,\ 
 a_{2} = 9.839673,\ 
a_{3} = 4.278, \ a_{4} = 1.
\end{align*}
We impose the condition
\begin{equation}
\frac{a_0}{a_1-a_0}c <r<1. \label{cond21}  
\end{equation}
We apply Lemma \ref{Stechkin1}, using \myref{eq-isolate4} for $k=1$, and \myref{eq-isolate7} otherwise:
\begin{multline*}
S_{1}(\s,\g_0)
\le - \frac{a_1}{\sigma-\beta_0} + 2\alpha_1 \sum_{k=0,2,3,4} a_k
\\ - \sum_{k=0,2,3,4} a_k \Big(\frac{\s-\b_0}{(\s-\b_0)^2+(k-1)^2\g_0^2} + \frac{\s-\b_0}{(\s-\b_0)^2+(k+1)^2\g_0^2} \Big). 
\end{multline*}
We apply Lemma \ref{lem-pole}, using \myref{pole0} for $k=0$, and \myref{pole1} otherwise:
\[
S_{3}(\s,\g_0)
\le \frac{a_0}{\s-1} + a_0\alpha_{20}
+ \sum_{k=1}^4 \frac{a_k (\s-1)}{(\s-1)^2+k^2\g_0^2} 
+ \alpha_{21} \sum_{k=1}^4 a_k. 
\]
We apply Lemma \ref{boundGamma<1}, using equations \myref{gamma0} for $k=0$, and \myref{gamma1} otherwise:
\[
S_{4}(\s,\g_0) \le 
\sum_{k=0}^4 a_k \Big( -\frac{1-\kappa}2 \log \pi + d(k) \Big) n_K 
\le 
0.
\]
Together with \myref{positive-explicit-formula}, \myref{eq-trig1}, and the above inequalities, we deduce
 \begin{multline}\label{trig3}
 0 \le 
  \frac{a_0}{\s-1}- \frac{a_1}{\sigma-\beta_0}
+   \frac{a_1(\s-1)}{(\s-1)^2+\g_0^2}
- \frac{2a_0(\s-\b_0)}{(\s-\b_0)^2+\g_0^2}
- \frac{a_1(\s-\b_0)}{(\s-\b_0)^2+4\g_0^2}
\\+ \sum_{k=2,3,4} a_k \Big(\frac{\s-1}{(\s-1)^2+k^2\g_0^2}
-\frac{\s-\b_0}{(\s-\b_0)^2+(k-1)^2\g_0^2} - \frac{\s-\b_0}{(\s-\b_0)^2+(k+1)^2\g_0^2} \Big)
\\
+ \frac{1-\kappa}2 \Big(\sum_{k=0}^4 a_k \Big) \log d_K 
+2 \alpha_1 \sum_{k=0,2,3,4} a_k
 + a_0\alpha_{20}
+ \alpha_{21} \sum_{k=1}^4 a_k. 
 \end{multline}
For $k=2,3,4$, since $\g_0\in(d_1/\L,d_2/\L)$ we have
\begin{multline*}
\frac{\s-1}{(\s-1)^2+k^2\g_0^2}
-\frac{\s-\b_0}{(\s-\b_0)^2+(k-1)^2\g_0^2} - \frac{\s-\b_0}{(\s-\b_0)^2+(k+1)^2\g_0^2}
\\ \le \Big( \frac{r}{r^2+k^2 d_1^2}
-\frac{r+c}{(r+c)^2+(k-1)^2d_2^2} - \frac{r+c}{(r+c)^2+(k+1)^2d_2^2} \Big)\frac1{\L} .
\end{multline*}
Since $r$ and $c$ satisfy \myref{cond21}, the same argument that gave \myref{bound-first-zeros} applies. Thus
\[
   \frac{a_1(\s-1)}{(\s-1)^2+\g_0^2} - \frac{a_0(\s-\b_0)}{(\s-\b_0)^2+\g_0^2}   
\]
decreases with $\g_0\in (d_1/\L,d_2/\L)$.
We obtain
\[
   \frac{a_1(\s-1)}{(\s-1)^2+\g_0^2} - \frac{a_0(\s-\b_0)}{(\s-\b_0)^2+\g_0^2}   
\le \Big(   \frac{a_1 r}{r^2+d_1^2} -  \frac{a_0 (r+c)}{(r+c)^2+d_1^2} \Big)\L. 
\]
We use the trivial bound for $\g_0\in (d_1/\L,d_2/\L)$:
\[
- \frac{a_0(\s-\b_0)}{(\s-\b_0)^2+\g_0^2}
- \frac{a_1(\s-\b_0)}{(\s-\b_0)^2+4\g_0^2}
 \le  
 - \Big( \frac{a_0(r+c)}{(r+c)^2+d_2^2} +  \frac{a_1(r+c)}{(r+c)^2+4d_2^2} \Big)\L.
\]
We deduce that \myref{trig3} becomes
$0 \le  \mathcal{E}(d_1,d_2,r,c)$. Here
 \begin{equation}
  \begin{split}
\label{trig31}
  \mathcal{E}(d_1,d_2,r,c)= &
  \frac{a_0}{r}- \frac{a_1}{r+c}
+    \frac{a_1 r}{r^2+d_1^2} -  \frac{a_0 (r+c)}{(r+c)^2+d_1^2}
- \frac{a_0(r+c)}{(r+c)^2+d_2^2} 
\\ &
-  \frac{a_1(r+c)}{(r+c)^2+4d_2^2}
+ \sum_{k=2,3,4} a_k \Big( \frac{r}{r^2+k^2 d_1^2}
-\frac{r+c}{(r+c)^2+(k-1)^2d_2^2} 
\Big. \\ & \Big.
- \frac{r+c}{(r+c)^2+(k+1)^2d_2^2} \Big)
+  \frac{1-\kappa}2\sum_{k=0}^4 a_k .
  \end{split}
 \end{equation}
Calculus gives that the above increases with $c$.
Thus the smallest value of $c$ satisfying the inequality \myref{trig31} is the root of $ \mathcal{E}(d_1,d_2,r,c)=0$.
We obtain
\begin{center}
\begin{tabular}{|c|c|c|c|}
\hline
$d_1$ & $d_2$ & $r$ & $1/c$ \\ 
\hline
$1.021$ & $2.374$ & 0.236 & $12.7301$ \\
\hline
\end{tabular}\end{center}
\begin{remark}
We explain our choice of $P$ in this section.
To simplify the analysis, we drop the terms depending on $d_1$ and $d_2$, and consider
 \[
  0=\frac{a_0}{\tilde{r}}- \frac{a_1}{\tilde{r}+\tilde{c}} + \sum_{k=0}^4 a_k \frac{1-\kappa}2 .
 \] 
As in remark \ref{remark1}, a similar analysis leads to
\[\tilde{c} = \frac{(a_1-a_0)\tilde{r}- \left(\frac{1-\kappa}2 \sum_{k=0}^4 a_k\right)  \tilde{r}^2 }{a_0+ \left(\frac{1-\kappa}2 \sum_{k=0}^4 a_k\right)  \tilde{r}}
 \]
with
\[
\tilde{r} = \frac{a_0\left(\sqrt{\frac{a_1}{a_0}}-1\right)}{\frac{1-\kappa}2 \sum_{k=0}^4 a_k}. 
\]
We set $d_1=1$, $d_2=2.5$, and run a Maple computation to determine which $a_k$'s make the root $c$ of $\mathcal{E}(1,2.5,\tilde{r},c)=0$ as small as possible. 
\end{remark}
\subsection{Case 4: Zero-free region when $0 < \g_0 \le \frac{d_1}{\L}$ with $d_1=1.021$.}\label{lower-region}
In this case we consider 
\[
S(\s,0)=f(\s,0),\]
where
\begin{multline*}
f(\s,0)=
-\sideset{}{'} \sum_{\beta \ge \frac{1}{2}} \left(F(\s,{\varrho})-\kappa F(\s_1,{\varrho})\right)  
+ \frac{1-\kappa}2\log d_K
\\+F(\s,1)-\kappa F(\s_1,1)  + \Re\left( \frac{\gamma'_K}{\gamma_K}(\s) -\kappa \frac{\gamma'_K}{\gamma_K}(\s_1) \right).
\end{multline*}
Using \myref{eq-isolate8}, \myref{pole0}, and \myref{gamma0}, we obtain
\begin{multline*}
f(\s,0)
\le 
- 2 \frac{\sigma-\beta_0}{(\sigma-\beta_0)^2+\g_0^2} + 2\alpha_1  
+ \frac{1-\kappa}2\log d_K
\\+\frac{1}{\s-1} + \alpha_{20}
+ \Big(  -\frac{1-\kappa}2 \log \pi + d(0) \Big) n_K.
\end{multline*}
The coefficient of $n_K$ is negative and may be dropped.
Together with the fact that $f(\s,0)\ge 0$, we deduce that, for $\g_0<\frac{d_1}{\L}$,
\[
0 \le 
\frac{1}{r}- 2 \frac{r+c}{(r+c)^2+d_1^2}   
+ \frac{1-\kappa}2
+\frac{ 2\alpha_1 + \alpha_{20}}{\L},
\]
which for $\L$ asymptotically large gives
\[
0 \le 
\frac{1}{r}- 2 \frac{r+c}{(r+c)^2+d_1^2}   
+ \frac{1-\kappa}2.
\]
We solve:
\[
c \ge \frac{-\frac{1-\kappa}2 r^2 + \sqrt{ r^2 - d_1^2(1+\frac{1-\kappa}2r)^2 } }{1+\frac{1-\kappa}2 r}, 
\]
and find 
\begin{center}
\begin{tabular}{|c|c|c|}
\hline
 $d_1$ & $r$ & $1/c$ \\ 
\hline
$1.021$ & $2.1426\ldots$ & $12.5494$ \\ \hline
$1/4$ & $1.5344\ldots$ & $1.6918 $ \\ \hline
$1/1.9996\ldots $ & $1.644$ & $1.9997$  \\ \hline
\end{tabular}\end{center}
The two last rows justify the regions announced in \myref{zfrD3} and \myref{zfrD4}. 
\subsection{Case 5: The case of real zeros}\label{real-region}\label{two-real}
Consider $\b_1$ and $\b_2$, two real zeros with $\b_1 \le \b_2$.
We isolate both of them from the sum over the zeros, and use the trivial inequality: 
\[
-\frac1{\s-\b_1}-\frac1{\s-\b_2} \le -\frac2{\s-\b_1}.
\]
It follows from $f(\s,0)\ge0$ that   
\[
0\le f(\s,0)  \le - \frac2{\s-\b_1} + 2\alpha_1
+ \frac{1-\kappa}2\log d_K
+ \frac1{\s-1} + \alpha_{20}.
\]
We write $1-\b_1=\frac{c_1}{\L}$ and obtain for $\L$ sufficiently large  
\[
 0 \le 
\frac1r   -  \frac2{r+c_1}
 +  \frac{1-\kappa}2. 
\]
The largest value for $c_1$ is given by $\ds{\tfrac{r-\tfrac{1-\kappa}2 r^2}{1+\tfrac{1-\kappa}2 r}}$ and this expression is maximized for 
$\ds{ r=2\tfrac{\sqrt{2}-1}{1-\kappa}}$:
\begin{center}
\begin{tabular}{|c|c|}
\hline
 $r$ & $1/c_1$ \\ 
\hline 
$1.4986\ldots $ &  $1.6110 $ \\
\hline
\end{tabular}\end{center}
This proves that there is at most one zero in the region \[\Re s \ge 1- \frac{1}{1.62 \L} \text{  and  } \Im s=0.\]
\subsection{Conclusion.}
Observe that $\max(12.7305, 12.7301 , 12.5494 ) = 12.7305$.
Combining the results proven in Sections \ref{superior-region}, \ref{middle-region}, \ref{lower-region}, and \ref{real-region}, we deduce that $\zeta_K(s)$ has at most one zero in the region:
\[
\Re s \ge 1-\frac1{12.74 \log d_K},\ |\Im s| \le 1.
\]
Moreover, it follows from Section \ref{real-region} that this zero, if it exists, is real and simple. 
This completes the proof of \myref{zfrD2} of Theorem \ref{ZFR}.


\begin{thebibliography}{}
\bibitem{Ven} M. Einsiedler, E. Lindenstrauss, P. Michel, A. Venkatesh, 2011,
``Distribution of periodic torus orbits and Duke's theorem for cubic fields'',
Annals of Math 173, No. 2, pp. 815-885.
\bibitem{For} K. Ford, 
``Zero-free regions for the Riemann zeta function'', 
Number theory for the millennium, II (Urbana, IL, 2000), 25Ð56, A K Peters, Natick, MA, 2002.
\bibitem{HB} D.R. Heath-Brown, 1992, 
``Zero-free regions for Dirichlet $L$-functions, and the least prime in an arithmetic progression'', 
Proc. London Math. Soc. 64, No. 3, pp. 265-338. 
\bibitem{Kad1} H. Kadiri, 2005,
``Une r\'egion explicite sans z\'eros pour la fonction $\zeta$ de Riemann'', 
Acta Arithmetica, Vol. 117.4, pp. 303-339. 
\bibitem{Kad2} H. Kadiri, 2005,
``Explicit zero-free regions for Dirichlet $L$-functions'', 
preprint,\\ http://arxiv.org/abs/math/0510570. 
\bibitem{NK1} H. Kadiri, N. Ng, 
``Explicit zero density theorems for Dedekind zeta functions'', 
submitted.
\bibitem{NK2} H. Kadiri, N. Ng, 
``A new bound for the least prime ideal in the Chebotarev density theorem'', 
preprint.
\bibitem{LMO} J.C. Lagarias, H.L. Montgomery, A.M. Odlyzko, 1979,
``A bound for the least prime ideal'', 
Invent. math 54, pp. 271-296. 
\bibitem{LO} J.C. Lagarias, A.M. Odlyzko, 1997,
``Effective versions of the Chebotarev density theorem'', 
in Algebraic Number Fields, L-functions and Galois Properties
(A. Frolich, ed.), pp. 409-464.  New York, London: Academic Press. 
\bibitem{Li} X. Li, 2010,
``The smallest prime that does not split completely in a number field'', 
preprint, http://arxiv.org/abs/1003.5718. 
\bibitem{MC1} K.S. McCurley, 1984,
``Explicit Zero-Free Regions for Dirichlet $L$-functions'', 
Journal of Number Theory 19, pp. 1-32. 
\bibitem{MC2} K.S. McCurley, 1984, 
``Explicit estimates for the error term in the prime number theorem for arithmetic progressions'', 
Mathematics of Computation 42, No. 165, pp. 265-285.
\bibitem{Ros} J.B. Rosser, 1941,
``Explicit bounds for some functions of prime numbers'', 
American Journal of Mathematics, Vol. 63, No. 1, pp. 211-232.
\bibitem{RS1} J.B. Rosser \& L. Schoenfeld, 1962, 
``Approximate formulas for some functions of prime numbers'', 
Illinois journal of Math. 6, pp. 64-94.
\bibitem{RS2} J.B. Rosser \& L. Schoenfeld, 1975, 
``Sharper bounds for the Chebyshev functions $\theta(x)$ and $\psi(x)$'', 
Math. Comp. 29, No. 129, pp. 243-269.
\bibitem{Sch} L. Schoenfeld, 1976, 
``Sharper bounds for the Chebyshev functions $\theta(x)$ and $\psi(x)$'', 
Math. Comp. 30, No. 134, pp. 337-360.
\bibitem{Sta} H.M. Stark, 1974, 
``Some effective cases of the Brauer-Siegel theorem'', 
Inventiones Math. 23, pp. 135-152.
\bibitem{St1} S.B. Stechkin, 1970,  
``Zeros of the Riemann zeta-function'', 
Mat. Zametki 8, pp. 419-429 (Russian; English translation in Math. Notes, 8, 1970, pp. 706-711.
\bibitem{St2}  S.B. Stechkin, 1989, 
``Rational inequalities and zeros of the Riemann zeta-function'', 
Trudy Math. Inst. Steklov, 189, pp 110-116 ; English translation in Proc. Steklov Inst. Math, AMS Translations Series, 4, 1990, pp. 127-134.
\bibitem{Xyl}  T. Xylouris, 2009, 
``On Linnik's constant'', 
http://arxiv.org/abs/0906.2749.
\end{thebibliography}
\end{document}